\documentclass[11pt, letterpaper, oneside]{amsart}

\usepackage{amssymb}
\usepackage{indentfirst}
\usepackage{enumerate}
\usepackage{url}
\usepackage{hyperref}

\theoremstyle{plain}
\newtheorem{thm}{Theorem}[section]

\newtheorem{prop}[thm]{Proposition}
\newtheorem{lemma}[thm]{Lemma}

\theoremstyle{definition}
\newtheorem{defi}[thm]{Definition}

\theoremstyle{remark}
\newtheorem{remark}[thm]{Remark}
\newtheorem{ep}[thm]{Example}

\newcommand{\ZZ}{\ensuremath{\mathbb Z}}

\newcommand{\RR}{\ensuremath{\mathbb R}}
\newcommand{\h}{\ensuremath{\mathfrak{h}}}
\newcommand{\g}{\ensuremath{\mathfrak{g}}}

\newcommand{\li}{\ensuremath{L_{\infty}}}
\newcommand{\cM}{\mathcal{M}}

\newcommand{\linfty}{\li}
\newcommand{\vect}{\mathfrak{X}}
\newcommand{\lin}{\mathrm{lin}}
\newcommand{\comm}{\mathrm{comm}}
\newcommand{\pdiff}[2]{\frac{\partial #1}{\partial #2}}

\newcommand{\tot}{\mathrm{tot}}
\newcommand{\defequal}{:=}
\newcommand{\lie}{\mathcal{L}}
\DeclareMathOperator{\Sh}{Sh}
\DeclareMathOperator{\sgn}{sgn}
\DeclareMathOperator{\ad}{ad}
\newcommand{\Sym}{\mathrm{S}}

\newcommand{\hatS}{\hat{\Sym}}
\DeclareMathOperator{\Hom}{Hom}
\DeclareMathOperator{\End}{End}
\DeclareMathOperator{\Der}{Der}

\DeclareMathOperator{\CDO}{CDO}
\newcommand{\innerprod}[2]{\langle #1, #2 \rangle}

\begin{document}

\title{$L_{\infty}$-algebra actions}

\author{Rajan Mehta}
\address{Department of Mathematics\\ Pennsylvania State University\\ State College, PA 16802}
\email{mehta@math.psu.edu}
\author{Marco Zambon}
\address{ 
Universidad Aut\'onoma de Madrid (Departamento de Matem\'aticas) and ICMAT(CSIC-UAM-UC3M-UCM),
Campus de Cantoblanco,
28049 - Madrid, Spain}
\email{marco.zambon@uam.es,marco.zambon@icmat.es} 

\date{}

\begin{abstract}
We define the notion of action of an $\li$-algebra $\g$ on a graded manifold $\cM$, and show that 
  such an action corresponds to a homological vector field on $\g[1] \times \cM$ of a specific form.  
This generalizes the correspondence between Lie algebra actions on manifolds and transformation Lie algebroids. In particular, we consider actions of $\g$ on a second $\li$-algebra, leading to a notion of ``semidirect product'' of $\li$-algebras more general than those we found in the literature.
\end{abstract}
\thanks{  
}

\maketitle
\tableofcontents

\section{Introduction}\label{intro}
Given a Lie algebra $\g$ and  a manifold $M$, an infinitesimal action of $\g$ on $M$ can be encoded by a \emph{transformation Lie algebroid} structure on the vector bundle $\g \times M$  over $M$. Whereas the definition of the action involves the infinite-dimensional space of vector fields on $M$, the transformation Lie algebroid is a finite-dimensional object.  {Transformation Lie algebroid structures on $\g \times M$ are characterized by the property} that the projection $\g \times M \to \g$ is a Lie algebroid morphism. 

In this note we extend this result by passing from ordinary geometry to the setting of graded geometry. That is, given an $\li$-algebra $\g$ and a graded manifold $\cM$, we define the notion of $\li$-action of $\g$ on $\cM$ (Definition \ref{def:liaction}). We show that such an action can be encoded by a homological vector field on $\g[1] \times \cM$, {and that the homological vector fields arising in this way are characterized by the property that} the projection map $\g[1] \times \cM \to \g[1]$ is a $Q$-manifold morphism (see Theorem
\ref{charvf}). In other words, $\g \times \cM$ becomes a \emph{transformation $\li$-algebroid} over $\cM$.

The zero component of the $\li$-action amounts to a homological vector field $Q_{\cM}$, giving $\cM$ the structure of a $Q$-manifold. On the other hand, if $\cM$ is already a $Q$-manifold, then we may consider $\li$-actions whose zero component agrees with the fixed homological vector field on $\cM$. 

Our main example is the case where $\cM$ is the $Q$-manifold associated to another $\li$-algebra $\h$. In this case, transformation $\li$-algebroids are just $\li$-algebra extensions of $\g$ by $\h$.   Special cases include $\li$-modules, representations up to homotopy, and the ``adjoint representation'' of an $\li$-algebra (\S \ref{sec:lialg}).

Another geometrically interesting example arises when $\cM$ is a degree $1$ $N$-manifold (\S \ref{sec:LA}). In this case, $(\cM,Q_{\cM})$ is a degree $1$ $NQ$-manifold, which is equivalent to a Lie algebroid.

One possible application of this work is the integration of certain $Q$-manifold extensions, as follows.
Given an $\li$-algebra $\g$ and a graded manifold $\cM$,
by Theorem \ref{charvf}, \emph{any} $Q$-manifold structure on $\g[1] \times \cM$ projecting to the one on $\g[1]$   arises from an $\li$-action of $\g$ on $\cM$. In some cases  this action can be integrated, and the integrated action can be used to construct the   integration of the $Q$-manifold $\g[1] \times \cM$.
This integration procedure has been worked out by Sheng and Zhu in \cite{shengzhusemi, shengzhuInt} when $\g$ is  Lie algebra, 
$\cM$ is a 2-term complex of vector spaces, and the $\li$-action is linear (a representation up to homotopy). It would be 
 particularly interesting to apply this procedure to integrate extensions of $\li$-algebras.

We expect our results to hold more generally when $\g[1]$ is replaced by an arbitrary $Q$-manifold. Such a result would be relevant to the integration of Courant algebroids \cite{lbs:integration, mt:double, CCShengHigherExt}.

\noindent\textbf{Conventions:}  Multibrackets of $\li$-algebras are  always denoted by $[\dots]_k$. They are graded skew-symmetric, of degree $2-k$. Multibrackets of $\li[1]$-algebras are   denoted by $\{\dots\}_k$. They are graded  symmetric, of degree  one.\\

\noindent\textbf{Acknowledgements:} 
R. Mehta acknowledges support of CMUP at the Universidade do Porto, where this project was initiated.
M. Zambon was partially supported by CMUP and FCT (Portugal) through the programs POCTI, POSI and Ciencia 2007, by projects PTDC/MAT/098770/2008 and PTDC/MAT/099880/2008 (Portugal) and by  projects MICINN RYC-2009-04065 and MTM2009-08166-E (Spain).  

\section{Review: \texorpdfstring{$L_{\infty}$}{L-infinity}-algebras as \texorpdfstring{$Q$}{Q}-manifolds}\label{sec:der}

In this section we review $\li$-algebras and their description in terms of $Q$-manifolds (\cite{LadaMarkl},\cite[\S 1.1]{FloDiss}).

Let $V=\oplus_{n\in \ZZ}V_n$ be a graded vector space. 
Letting $TV:=\oplus_{i\ge 0}T^i V=\RR \oplus V \oplus (V\otimes V) \oplus \cdots$ be the tensor algebra over $V$, we define the graded symmetric algebra over $V$ as 
$$SV:=TV/\langle v \otimes v' -(-1)^{|v||v'|}v' \otimes v \rangle,$$ 
where $|v|$ denotes the degree of a homogeneous element $v \in V$. Similarly, the graded skew-symmetric algebra over $V$ is defined as
$$\wedge V:=TV/ \langle v \otimes v' +(-1)^{|v||v'|}v' \otimes v \rangle.$$  
We introduce the notation $S^+ V:=\oplus_{i\ge 1}S^i V$ and $\wedge^+ V:=\oplus_{i\ge 1}\wedge ^i V$.

For $j\ge 1$, let $S_j$ denote the permutation group in $j$ letters. For any homogeneous elements $v_1,\dots,v_j \in V$ and $\tau\in S_j$, the \emph{Koszul signs} $\epsilon(\tau)$ and $\chi(\tau)$ are defined by the equations
$$v_{\tau_1}\cdots v_{\tau_j}=\epsilon(\tau)v_{1}\cdots v_{j}$$
and 
$$v_{\tau_1}\wedge \cdots \wedge v_{\tau_j}=\chi(\tau)v_{1} \wedge \cdots \wedge v_{j},$$
where the products are taken in $SV$ and $\wedge V$ respectively.
The two signs are related by the equation $\chi(\tau)=\sgn(\tau)\epsilon(\tau)$. We stress that $\epsilon(\tau)$ and $\chi(\tau)$ depend on the $v_i$ in addition to $\tau$, although the notation does not make this dependence explicit.

Recall that $\tau\in S_n$ is called an \emph{$(i, n-i)$-unshuffle} if it satisfies the inequalities $\tau(1)<\dots<\tau(i)$ and $\tau(i+1)<\dots<\tau(n)$. The set of $(i, n-i)$-unshuffles is denoted by $\Sh(i, n-i).$ Following \cite[Def. 2.1]{LadaMarkl}\footnote{In \cite{LadaMarkl}, the multi-brackets have degree $k-2$, instead of $2-k$. They are related to ours by simply inverting the grading of the graded vector space $V$.} and \cite[Def. 5]{KajSta}, we make the following definitions.  

\begin{defi}\label{li}
An \emph{$\li$-algebra} is a graded vector space $V$ equipped with a collection of linear maps $[\cdots]_k \colon \wedge ^kV\longrightarrow V$ of degree $2-k$, for $k\ge1$, such that
$$\sum_{i=1}^n (-1)^{i(n-i)}\sum_{\tau\in \Sh(i,n-i)}\chi(\tau)[[v_{\tau(1)}, \dots,v_{\tau(i)}]_i,v_{\tau(i+1)}, \dots, v_{\tau(n)} ]_{n-i+1}=0$$
for every collection of homogeneous elements $v_1, \dots, v_n \in V$. 
\end{defi}  

\begin{defi}\label{li1}
An \emph{$\li[1]$-algebra}\footnote{The name ``$\li[1]$-algebra'' is borrowed from  \cite {FloDiss}} is a graded vector space $V$ equipped with a collection of linear maps $\{\cdots\}_k \colon S^kV\longrightarrow V$ of degree $1$, for $k\ge1$, such that 
$$\sum_{i=1}^n \sum_{\tau\in \Sh(i,n-i)}\epsilon (\tau)\{\{v_{\tau(1)}, \dots,v_{\tau(i)}\}_i,v_{\tau(i+1)}, \dots, v_{\tau(n)} \}_{n-i+1}=0$$
for every collection of homogeneous elements $v_1, \dots, v_n \in V$.
\end{defi}

There is a bijection \cite[Rem. 2.1]{Vor} between \li-algebra structures on a graded vector space $V$ and $\li[1]$-algebra structures on $V[1]$, the graded vector space defined by $(V[1])_i:=V_{i+1}$. The multibrackets are related by the \emph{d\'ecalage isomorphism}
\begin{equation}\label{deca}
 (\wedge^n V)[n] \cong S^n(V[1]),\;\; v_1\cdots v_n \mapsto v_1\cdots v_n\cdot (-1)^{(n-1)|v_{1}|+\dots+2|v_{n-2}|+|v_{n-1}|}.
 \end{equation}
We  refer to an $\li[1]$-algebra structure as a \emph{Lie$[1]$ algebra}  when only $\{\cdots\}_2$ is non-trivial, and we use the terminology \emph{DGL$[1]$-algebra} when only $\{\cdots\}_1$ and $\{\cdots\}_2$ are non-trivial.

We provide the notion of $\li[1]$-algebra morphism \cite[Def. 6]{KajSta}
\cite[Def. 1.7]{FloDiss} only in a special case.

\begin{defi}\label{li1mor}  
Let $V$  be an {$\li[1]$-algebra} with multibrackets   $\{\cdots\}$, and let $W$ be a DGL$[1]$-algebra with brackets $\textbf{d}$ and $\textbf{\{}\cdot,\cdot\textbf{\}}$.
An \emph{$\li[1]$-algebra morphism} from $V$ to $W$, denoted $V \rightsquigarrow W$, is a degree $0$ linear map $\phi \colon S^+ V\to W$ such that  for all $n \geq 1$
 \begin{equation}\label{li1morph}
 	\begin{split}
 &\sum_{i=1}^n \sum_{\tau\in \Sh(i,n-i)}\epsilon (\tau) \phi_{n-i+1}(\{v_{\tau(1)}, \dots,v_{\tau(i)}\}_i,v_{\tau(i+1)}, \dots, v_{\tau(n)}) \\
=&\textbf{d}\phi_n(v_{1}, \dots,v_{n})
+\frac{1}{2}
\sum_{j=1 }^{n-1}  \sum_{\tau\in \Sh(j,n-j)} \epsilon (\tau)
\textbf{\{}\phi_{j}(v_{\tau(1)}\;,\; \dots,v_{\tau(j)}), 
\phi_{n-j}(v_{\tau(j+1)}, \dots,v_{\tau(n)}
)\textbf{\}}.		
 	\end{split}
 \end{equation}
for every collection of homogeneous elements $v_1, \dots, v_n \in V$.
\end{defi}

In the setting of the above definition, a \emph{curved} $\li[1]$-algebra morphism $V \rightsquigarrow W$ consists of a degree $0$ linear map $\phi: SV \to W$ satisfying, for $n \geq 0$, a variant of 
\eqref{li1morph} where the index $j$ on the right side of the equation runs from $0$ to $n$. A key point is that nontrivially curved morphisms exist even though we are not considering $\li[1]$-algebras equipped with $0$-brackets.

\begin{remark}
If $\phi: SV \to W$ is a degree $0$ linear map, then the zero component $\phi_0 \colon \RR \to W_0$ gives rise to an element $\phi_0(1)\in W_0$, which by abuse of notation we denote by $\phi_0$. The curved variant of \eqref{li1morph} for $n=0$ then reads 
$0=\textbf{d}\phi_0+\frac{1}{2}\textbf{\{}\phi_0,\phi_0\textbf{\}}.$
In other words, if $\phi$ is a curved $\li[1]$-algebra morphism, then $\phi_0$ is a Maurer-Cartan element of $W$.
\end{remark}

The notion of (curved) $\li$-algebra morphism into a DGLA corresponds to Definition \ref{li1mor} (and the following paragraph) via the d\'ecalage isomorphism \eqref{deca}.
 
The following fact, whose proof follows easily from the above, is implicit in \cite{K}:
\begin{lemma}\label{lemma:curved}
Let $V$ be an $L_{\infty}[1]$-algebra, let $(W,\textbf{\{}\cdot,\cdot\textbf{\}})$  be a Lie$[1]$ algebra, and let $\phi \colon SV\to W$ be a degree zero linear map.  Then
$\phi$ is a curved $\li[1]$-algebra morphism from $V$ to $W$ if and only if 
$\textbf{\{}\phi_0,\phi_0\textbf{\}}=0$ and $\phi_+ \defequal \sum_{i\ge 1} \phi_i$ is an $\li[1]$-algebra morphism from $V$ to the DGL$[1]$-algebra $(W,\textbf{\{}\phi_0,\cdot\textbf{\}},\textbf{\{}\cdot,\cdot\textbf{\}})$.
\end{lemma}

Now we turn to the supergeometric description of $\li$-algebras.
\begin{defi}\label{Qmfld}
A \emph{Q-manifold} is a graded manifold {(see, for example, \cite{cattaneo:super,rajthesis})} equipped with a 
\emph{homological} vector field, that is, a degree $1$ vector field $Q$ such that $Q^2 = \frac{1}{2}[Q,Q] = 0$.
\end{defi}

Given a \emph{finite dimensional} graded vector space $V$, there is a bijection between $L_{\infty}[1]$-algebra structures on $V$ and formal homological vector fields $Q$ on $V$ vanishing at the origin.  (The modifier ``formal'' refers to the fact that $Q$ is a derivation of 
the formal power series on $V$.) The correspondence is given by Voronov's \emph{higher derived bracket} construction \cite[Ex. 4.1]{Vor}. Specifically, the $\li[1]$-multibrackets induced by a formal homological vector field $Q$ are defined by
 \begin{equation}\label{eq:derbr}
 \{v_1,\dots,v_n\}_n = [[[Q,\iota_{v_1}],\dots],\iota_{v_n}]|_0.
\end{equation}
Here  ``$|_0$'' denotes evaluation at the origin in $V$, and $\iota_v$ is the (formal) vector field that acts on linear functions $\xi \in V^*$ by $\langle v, \xi \rangle$.

\section{An alternative characterization of \texorpdfstring{$L_{\infty}$}{L-infinity}-morphisms}

The goal of this section is to obtain a useful formula characterizing curved $\linfty[1]$-morphisms $U \rightsquigarrow V$ in the case where $U$ is finite-dimensional (but otherwise has an arbitrary $\linfty[1]$-algebra structure) and where $V$ is a Lie$[1]$ algebra (but is possibly infinite-dimensional). At first, this may seem like a peculiar set of assumptions, but we will see in \S	 \ref{sec:actionq} that this setting is exactly the one that arises in the context of actions on graded manifolds.

Let $U$ be a finite-dimensional graded vector space, and let $V$ be a (possibly infinite-dimensional) graded vector space. Denote by $\hatS(U^*)$ the algebra of formal power series on $U$. We may consider elements of $\hatS(U^*) \otimes V$ to be $V$-valued formal power series on $U$. 

Let $\Hom(SU,V)$ denote the graded vector space of linear (not necessarily degree-preserving) maps from $SU$ to $V$. To any $X \in \hatS(U^*) \otimes V$, we may associate an element $\phi_X \in \Hom(SU, V)$, given by
\begin{equation}\label{eq:phi}
     \phi_X (e_1\cdots e_n) = [[\cdots [X, \iota_{e_1}],\dots],\iota_{e_n}]|_0 
\end{equation}
for $e_i \in U$, where $0$ denotes the origin in $U$. By definition, the brackets on the right side of \eqref{eq:phi} are given by $[X, \iota_e] = -(-1)^{|X||e|} \iota_e X$. 
(One should think of this bracket as the Lie bracket of vector fields on the graded manifold $U\times V$.)
\begin{lemma}\label{lemma:phiiso}
     The map $\phi: X \mapsto \phi_X$ is a degree-preserving isomorphism from $\hatS(U^*) \otimes V$ to $\Hom(\Sym U,V)$. 
\end{lemma}
\begin{proof}
We first observe that $\phi$ is the direct product of maps $\phi_k: \Sym^k(U^*) \otimes V \to \Hom(\Sym^k U, V)$, so it suffices to prove that each $\phi_k$ is an isomorphism. Next, consider the case $V = \RR$. In this case, it is fairly easy to see (for example, in terms of a basis) that $\phi$ induces isomorphisms from $\Sym^k(U^*)$ to $\Hom(\Sym^k U, \RR) = (\Sym^k U)^*$.

Returning to the general case, we can identify $\Hom(\Sym^k U, V)$ with $(\Sym^k U)^* \otimes V$ in the obvious way. The result then follows from right-exactness of the tensor product.
\end{proof}

Now suppose that $V$ is equipped with a Lie$[1]$ algebra structure, with binary bracket $\{\cdot,\cdot\}_V$. Then we may extend the bracket to $\hatS(U^*) \otimes V$ by graded linearity:
\begin{equation}
     \{f \otimes v, f' \otimes v'\}_V = (-1)^{|f'||v|} ff' \otimes \{v,v'\}_V. \label{eqn:binsuv}
\end{equation}
The relationship between the operation \eqref{eqn:binsuv} and the map $\phi$ is described by the following lemma, which can be proven directly for $X = f \otimes v$.
\begin{lemma}\label{lemma:binsuv}
The following equation holds for all $X, X' \in \hatS(U^*) \otimes V$ and $e_i \in U$:
\begin{multline*}
	     \phi_{\{X,X'\}_V}(e_1 \cdots e_n) = \\
\sum_{j=0}^n \sum_{\tau \in \Sh(j,n-j)} \epsilon(\tau)(-1)^{|X'|(|e_{\tau(1)}| + \cdots + |e_{\tau(j)}|)} \left\{ \phi_X (e_{\tau(1)} \cdots e_{\tau(j)}), \phi_{X'}(e_{\tau(j+1)} \cdots e_{\tau(n)}) \right\}_V.
\end{multline*}
\end{lemma}

Suppose further that $U$ is equipped with an arbitrary $\li[1]$-algebra structure, for which the corresponding homological vector field is $Q_U$. For any degree $0$ element $X \in \hatS(U^*) \otimes V$, we may ask whether $\phi_X \in \Hom(\Sym U , V)$ is a curved $\linfty[1]$-algebra morphism. The following statement expresses this property directly in terms of $X$.
\begin{prop}\label{prop:limor}Let $U$ be a finite dimensional  $\li[1]$-algebra,  $V$  a Lie$[1]$ algebra, and  $X \in \hatS(U^*) \otimes V$ of degree zero.
   Then $\phi_X$ is a curved $\linfty[1]$-algebra morphism $U \rightsquigarrow V$ if and only if 
\begin{equation*}
     Q_U(X) = \frac{1}{2}\{X,X\}_V.
\end{equation*}  
Here, the left side is defined by $Q_U(f \otimes v)=Q_U(f)\otimes v$.
\end{prop}
\begin{proof}
     In this case, the curved variant of \eqref{li1morph} reduces to
\begin{multline}\label{eqn:dglmorph}
          \sum_{i=0}^n \sum_{\tau \in \Sh(i,n-i)} \epsilon(\tau) \phi_X (\{e_{\tau(1)}, \dots,  e_{\tau(i)}\} e_{\tau(i+1)} \cdots e_{\tau(n)} ) = \\
 \frac{1}{2} \sum_{j=0}^n \sum_{\tau \in \Sh(j,n-j)} \epsilon(\tau) \left\{ \phi_X (e_{\tau(1)} \cdots e_{\tau(j)}), \phi_{X}(e_{\tau(j+1)} \cdots e_{\tau(n)}) \right\}_V
\end{multline}
for all $n\geq 0$ and $e_1, \dots, e_n \in U$. Since $X$ is of degree $0$, Lemma \ref{lemma:binsuv} implies that the right side of \eqref{eqn:dglmorph} is equal to $\frac{1}{2} \phi_{\{X,X\}_V}(e_1 \cdots e_n)$.

On the other hand, by repeatedly applying the Jacobi identity and using the fact that $X$ is of degree $0$, we have that
\begin{multline*}
          \phi_{Q_U(X)}(e_1 \cdots e_n) =[[\dots[[Q_U,X], \iota_{e_1}] \dots],\iota_{e_n}]|_0 \\
=  \sum_{i=0}^n \sum_{\tau \in \Sh(r,n-i)} \epsilon(\tau) \left[[[\cdots [Q_U, \iota_{e_{\tau(1)}}]\dots],\iota_{e_{\tau(i)}}]\;,\; [[\cdots [X, \iota_{e_{\tau(i+1)}}]\dots],\iota_{e_{\tau(n)}}] \right]\vert_0.
\end{multline*}
This equals the left side of \eqref{eqn:dglmorph}, as one sees using the identity $[\tilde{Q},\tilde{X}]|_0=[\tilde{Q}|_0,\tilde{X}]|_0$ for $\tilde{Q}\in
 S(U^*) \otimes U$ and
 $\tilde{X}\in S(U^*) \otimes V$ (recall that $|_0$ denotes evaluation at the origin in $U$).  Therefore, \eqref{eqn:dglmorph} holds for all $e_i \in U$ if and only if $\phi_{Q_U(X)} = \frac{1}{2}\phi_{\{X,X\}_V}$. The result then follows from Lemma \ref{lemma:phiiso}.
\end{proof}

\section{\texorpdfstring{$L_{\infty}$}{L-infinity}-actions on graded manifolds}\label{sec:actionq}

Let $\cM$ be a graded manifold, and let $U$ be a finite-dimensional graded vector space. We consider the algebra $\hatS(U^*) \otimes C(\cM)$ of functions on $U \times \cM$ that are formal in $U$ and smooth in $\cM$. Slightly abusing notation, we simply refer to this algebra as $C(U \times \cM)$. Similarly, we use $\vect(U \times \cM)$ to refer to the space of vector fields that are formal in $U$ and smooth in $\cM$. 

\begin{lemma}\label{lemma:cor1} Let $\cM$ be a graded manifold, and let $U$ be a finite-dimensional graded vector space.
     The following are in one-to-one correspondence:
\begin{enumerate}
     \item degree $1$ elements $X$ of $\vect(U \times \cM)$ that are annihilated by the projection onto $U$,
     \item degree-preserving maps $\phi_X \colon \Sym U \to \vect(\cM)[1]$, and
     \item (formal in $U$) graded manifold morphisms $\eta_X \colon U \times \cM \to T[1]\cM$ that cover the identity map on $\cM$.
\end{enumerate}
\end{lemma}
\begin{proof}
\emph{($1 \Leftrightarrow 2$)}  The subspace of $\vect(U \times \cM)$ consisting of vector fields that are annihilated by the projection onto $U$ can be naturally identified with $\hatS(U^*) \otimes \vect(\cM)$. The isomorphism of Lemma \ref{lemma:phiiso} (with $V = \vect(\cM)[1]$) provides a bijection between $\Hom(\Sym U,\vect(\cM)[1])$ and $\hatS(U^*) \otimes \vect(\cM)[1]$. In particular, a degree-preserving map from $\Sym U$ to $\vect(\cM)[1]$ corresponds to a degree $0$ element of $\hatS(U^*) \otimes \vect(\cM)[1]$, which in turn corresponds to a degree $1$ element of $\hatS(U^*) \otimes \vect(\cM)$.

\emph{($1 \Leftrightarrow 3$)} Given a degree $0$ element $X \in \hatS(U^*) \otimes \vect(\cM)[1]$, define a map $\eta_X: U \times \cM \to T[1]\cM$ by taking $\eta_X^* : \Omega(\cM) \to C(U \times \cM)$ to be the unique algebra morphism such that
$\eta_X^* \alpha = \innerprod{X}{\alpha}$ for all $\alpha \in \Omega^1(\cM)$ (recall that $\Omega(M)=C(T[1]M)$). Since $\eta_X^* f = f$ for $f \in \Omega^0(\cM)$, we see that $\eta_X$ covers the identity map on $\cM$. Conversely, one can see that any map from $U \times \cM$ to $T[1]\cM$, formal in $U$ and covering the identity map on $\cM$, is of the form $\eta_X$ for some $X \in \hatS(U^*) \otimes \vect(\cM)[1]$ of degree $0$.
\end{proof}

The space $\vect(\cM)$ of vector fields on $\cM$ is a graded Lie algebra, so $\vect(\cM)[1]$ is a Lie$[1]$ algebra with bracket $\{P,R\} = (-1)^{|P|} [P,R]$, where $|P|$ denotes the degree of $P$ in $\vect(\cM)$. The sign in this formula arises from the d\'{e}calage isomorphism.

Now, suppose that $U$ has the structure of an $\linfty[1]$-algebra, with homological vector field $Q_U$. Then we may also view $Q_U$ as a (horizontal) homological vector field on $U \times \cM$. The shifted tangent bundle $T[1]\cM$ is also a $Q$-manifold, where the homological vector field is the de Rham operator $d$.

For any degree $1$ element $X \in \hatS(U^*) \otimes \vect(\cM) \subseteq \vect(U \times \cM)$, let $\phi_X \in \Hom(\Sym U, \vect(\cM)[1])$ and $\eta_X : U \times \cM \to T[1]\cM$ be the corresponding objects as given by Lemma \ref{lemma:cor1}.

\begin{thm}\label{thm:equiv} Let $\cM$ be a graded manifold, and let $U$ be a finite-dimensional $\linfty[1]$-algebra, with homological vector field $Q_U$. Let $X \in \hatS(U^*) \otimes \vect(\cM)$ be of  degree $1$.
The following statements are equivalent:
\begin{enumerate}
     \item $Q_\tot \defequal X + Q_U$ is a homological vector field on $U \times \cM$.
     \item $\phi_X$ is a curved $\linfty[1]$-algebra morphism $U \rightsquigarrow \vect(\cM)[1]$.
     \item\label{three} $Q_\tot$ is $\eta_X$-related to $d$.
\end{enumerate}  
\end{thm}

\begin{proof}
\emph{($1 \Leftrightarrow 2$)} Proposition \ref{prop:limor} (with $V = \vect(\cM)[1]$) says that $\phi_X$ is a curved $\linfty[1]$-algebra morphism if and only if $Q_U(X) = \frac{1}{2} \{X,X\} = -\frac{1}{2} [X,X]$, where in the last step we have used the d\'{e}calage isomorphism to view $X$ as a degree $1$ vector field. Using the assumption $Q_U^2 = 0$ and the identity $\frac{1}{2}[X,X] = X^2$, we see that $(X + Q_U)^2 = 0$ if and only if $Q_U(X) = -\frac{1}{2}[X,X]$.

\emph{($1 \Leftrightarrow 3$)} To prove that $Q_\tot$ is $\eta_X$-related to $d$, it suffices (since $\Omega(\cM)$ is locally generated by functions and exact $1$-forms) to check that 
$$Q_\tot(\eta_X^*f) = \eta_X^* df\;\;\;\;\;\;\;\;\;\text{ and }\;\;\;\;\;\;\;\;\;Q_\tot(\eta_X^*(df)) = \eta_X^* d(df)$$ for all $f \in \Omega^0(\cM)$. The former equation holds automatically by the definition of $\eta_X$. In the latter equation, the right side obviously vanishes, and the left side equals  $Q_\tot^2  (\eta_X^*f) = Q_\tot^2(f)$. Since $Q_\tot^2$ automatically annihilates functions of $U$, we conclude that $Q_\tot^2 = 0$ if and only if $Q_\tot$ is $\eta_X$-related to $d$.
\end{proof}

\begin{defi}\label{def:liaction}     Let $\g$ be a finite-dimensional $\linfty$-algebra, and let $\cM$ be a graded manifold. An \emph{$\linfty$-action} of $\g$ on $\cM$ is a curved $\linfty$-algebra morphism $\g \rightsquigarrow \vect(\cM)$.
\end{defi}
The following is a direct application of Theorem \ref{thm:equiv}.
\begin{thm}\label{charvf}
Let $\g$ be a finite-dimensional $\linfty$-algebra, and let $\cM$ be a graded manifold. There is a one-to-one correspondence between 
     \begin{enumerate}
\item $\linfty$-actions of $\g$ on $\cM$
\item  homological vector fields on $\g[1] \times \cM$ for which the projection map $\g[1] \times \cM \to \g[1]$ is a $Q$-manifold morphism. 
\end{enumerate} 
Explicitly,  an $\linfty$-action is an element of $ \Hom(\wedge \g, \vect(\cM))
\cong \Hom(\Sym (\g[1]), \vect(\cM)[1])$; writing  $\phi_X$ for the latter, where $X\in S(\g[1]^*)\otimes \vect(\cM)$, the corresponding homological vector field is $Q_\tot :=X + Q_{\g[1]}$.
\end{thm}

\begin{remark}\label{rem:LA}
     Given an $\linfty$-action of $\g$ on $\cM$, one can geometrically interpret the associated homological vector field on $\g[1] \times \cM$ as giving the structure of a ``transformation $\linfty$-algebroid.'' Indeed, in the case where $\g$ is an ordinary Lie algebra and $M$ is an ordinary manifold, this construction reduces (up to a degree shift) to that of the usual transformation Lie algebroid. More precisely,
if $\g \to \vect(M)$ is a Lie algebra morphism, one can form the transformation Lie algebroid $\g \times M$ over $M$. The   homological vector field on $\g[1] \times M$ encoding the Lie algebroid structure is exactly $Q_{tot}$: the summand $X$   encodes the anchor map, while $Q_{\g[1]}$ encodes the Lie algebroid bracket. The anchor map $\g \times M \to TM$ is a Lie algebroid morphism, corresponding to
 the map  $\g[1] \times M\to T[1]M$, which by Thm. \ref{thm:equiv} (\ref{three}) preserves the homological vector fields.
\end{remark}

From Lemma \ref{lemma:curved}, we see that, given an $\linfty$-action   of $\g$ on $\cM$, the zero component of $\phi_X$  corresponds to a homological vector field $Q_\cM$ on $\cM$, and the higher components of $\phi_X$ then give a (noncurved) $\linfty[1]$-algebra morphism from $\g$ to the DGL$[1]$-algebra $(\vect(\cM)[1], \{Q_\cM, \cdot\}, \{\cdot,\cdot\})$. Equivalently, the zero component of 
the $\linfty$-action is $Q_\cM$ and its higher components give a (noncurved)
$\linfty$-algebra morphism from $\g$ to the DGLA $(\vect(\cM), -[Q_\cM, \cdot], [\cdot,\cdot])$. (The minus sign comes from the d\'ecalage isomorphism.) In many situations, $\cM$ already comes equipped with a homological vector field, motivating the following definition.
\begin{defi}\label{def:action}
Let $\g$ be a finite-dimensional $\linfty$-algebra, and let $(\cM, Q_{\cM})$ be a $Q$-manifold. An {$\linfty$-action} of $\g$ on $\cM$ is \emph{compatible with $Q_\cM$} if the zero component of $\phi_X$ corresponds to $Q_\cM$.
\end{defi}

The appropriate analogue of Theorem \ref{charvf} in this context is as follows:
\begin{thm}\label{charvfQM}
Let $\g$ be a finite-dimensional $\linfty$-algebra, and let $(\cM,Q_{\cM})$ be a $Q$-manifold. There is a one-to-one correspondence between 
     \begin{enumerate}
\item $\linfty$-actions of $\g$ on $\cM$ compatible with $Q_\cM$
\item  homological vector fields on $\g[1] \times \cM$ for which
\begin{equation}\label{eq:sesQM}
(\cM,Q_\cM)\to (\g[1] \times \cM, Q_{tot}) \to (\g[1],Q_{\g[1]})
\end{equation}
is a sequence of $Q$-manifold morphisms.
\end{enumerate}  \end{thm} 

\section{Equivalences}\label{sec:equiva}

 Let $\g$ be a finite-dimensional $\linfty$-algebra
 and let $\cM$ be a graded manifold. We use the notation of Theorem \ref{charvf}.
 
Consider the DGLA $(\vect(\g[1]\times \cM), [Q_{\g[1]}, \cdot], [\cdot,\cdot])$. The Maurer-Cartan equation for a degree $1$ element $X$ in this DGLA reads
$[Q_{\g[1]}, X]+\frac{1}{2}[X, X]=0$, so it is equivalent to
$X+Q_{\g[1]}$ being a homological vector field on $\g[1]\times \cM$.
In view of Theorem \ref{charvf}, therefore, $\li$-actions  of $\g$ on $\cM$
are in bijection with Maurer-Cartan elements $X$ of the above DGLA which are annihilated by the projection $\g[1]\times \cM \to \g[1]$.

Any nilpotent\footnote{This means that for every element of the DGLA there is a power of $\ad_{\lambda}=[\lambda,\cdot]$ annihilating it.}
 element $\lambda$ of degree $0$ in the DGLA  acts on the set of  Maurer-Cartan elements  \cite[\S1]{GoldmanMillson}, mapping a Maurer-Cartan element $X$ to
\begin{equation}\label{eqn:mcequivalent}
 X^{\lambda}:=e^{ad_{\lambda}}X+\frac{1-e^{\ad_{\lambda}}}{\ad_{\lambda}} [Q_{\g[1]}, \lambda]=e^{\ad_{\lambda}}(X+Q_{\g[1]})-Q_{\g[1]}.
\end{equation}
Notice that the action of nilpotent $\lambda$'s lying in $(\hatS(\g[1]^*)\otimes \vect(\cM))_0$  preserves the set of   Maurer-Cartan elements that are annihilated under the projection $\g[1]\times \cM \to \g[1]$, generating an equivalence relation there. 

\begin{defi}\label{def:equiv}
Two $\li$-actions   of $\g$ on $\cM$ are \emph{equivalent} if   the corresponding Maurer-Cartan elements are equivalent by the action of  nilpotent elements $\lambda\in (\hatS(\g[1]^*)\otimes \vect(\cM))_0$.\end{defi}

\begin{prop}\label{prop:iso}
Equivalent $\li$-actions of $\g$ on $\cM$ induce isomorphic $Q$-manifold structures on $\g[1]\times \cM$.
\end{prop}
\begin{proof} Consider an $\li$-action, with corresponding Maurer-Cartan element $X$ (so the associated homological vector field is $X+ Q_{\g[1]}$). For any nilpotent elements $\lambda \in (\hatS(\g[1]^*)\otimes \vect(\cM))_0$, we have by \eqref{eqn:mcequivalent} that
\begin{equation*}
  X^{\lambda}+ Q_{\g[1]}=e^{ad_{\lambda}}(X+Q_{\g[1]}).
\end{equation*}
In other words, the two homological vector fields are related by the diffeomorphism of the graded manifold $\g[1]\times \cM$ obtained by exponentiating the degree $0$ vector field $\lambda$.
\end{proof}

In practice, $\cM$ often comes equipped with a fixed homological vector field $Q_{\cM}$ (see Definition \ref{def:action}). If $\lambda$ is  any degree 0 element of $\hatS^+(\g[1]^*)\otimes \vect(\cM)$, one can show that $e^{ad_{\lambda}}$ always converges\footnote{The idea is to endow $\vect(\g[1]\times \cM)$ with a filtration induced by the  degree of polynomials  in $\hatS(\g[1]^*)$, cf. \cite[\S 1.4]{YaelZ}.} ($\hatS^+$ is defined analogously to the way $S^+$ is defined in \S\ref{sec:der}). Furthermore, as such a $\lambda$ vanishes on $\{0\} \times\cM$, its action 
preserves the set of Maurer-Cartan elements  whose restriction  to $\{0\} \times \cM$ agrees with $Q_\cM$. Hence we define:
\begin{defi}\label{def:equivQM}
Two $\li$-actions   of $\g$ on $\cM$ compatible with $Q_{\cM}$ are \emph{equivalent} if   the corresponding Maurer-Cartan elements are equivalent  by the action of  elements $\lambda\in (\hatS^+(\g[1]^*)\otimes \vect(\cM))_0$.\end{defi}

From Proposition \ref{prop:iso} we immediately obtain:
\begin{prop}\label{prop:isoQM} 
Compatible $\li$-actions of $\g$ on $\cM$ which are equivalent induce  $Q$-manifold structures on $\g[1]\times \cM$  which are isomorphic 
by   isomorphisms that commute with the maps appearing in \eqref{eq:sesQM}.
\end{prop}

\section{Extensions of \texorpdfstring{$L_{\infty}$}{L-infinity}-algebras}\label{sec:lialg}

Let $\g$ and $\h$ be finite-dimensional $\linfty$-algebras. Following Definition \ref{def:action}, we may consider formal $\linfty$-actions of $\g$ on $\h[1]$ that are compatible with the homological vector field $Q_{\h[1]}$. By formal $\linfty$-actions, we mean curved $\linfty$-algebra morphisms from $\g$ to the graded Lie algebra of formal vector fields on $\h[1]$. 

The following proposition is an application of Theorem \ref{charvfQM}, which {also holds in} the setting of formal actions and formal vector fields. Notice that $Q_\tot$ there vanishes at the origin,
since  the map $\h[1] \to \g[1] \times \h[1]$ is a $Q$-manifold morphism.

\begin{prop}\label{sdli}
There is a one-to-one correspondence between
\begin{enumerate}
\item compatible formal $\li$-actions of $\g$ on $\h[1]$ and
\item $\li$-algebra extensions\footnote{This means that the two arrows are strict morphisms of $\li$-algebras   and that the sequence is exact. Recall that a strict $\li[1]$-algebra morphism from $V$ to $W$ is a degree-preserving linear map
$V  \to W$ which intertwines the multi-brackets.}  $\h \to \g \times \h \to \g$.
\end{enumerate}
\end{prop}

\begin{remark}
Given a compatible formal $\linfty$-action of $\g$ on $\h[1]$ with associated $\linfty[1]$-algebra morphism $\phi: \g[1] \rightsquigarrow \vect(\h[1])[1]$, we may explicitly describe the corresponding multibrackets on $\g[1] \times \h[1]$ as follows. For $e_i \in \g[1]$, the $\g[1]$ component of $\{e_1, \dots, e_n\}$ coincides with the $\g[1]$-multibracket. The $\h[1]$ component is $\phi(e_1 \cdots e_n)|_0$.
 
In the case of mixed entries, we have
\begin{equation}\label{mixed}
\{e_1,\dots,e_n,f_1,\dots,f_k\}=[[\dots[\phi(e_1 \cdots e_n),\iota_{f_1} ],\dots],\iota_{f_k}]|_{0}\in \h[1]
\end{equation}
for $e_i\in \g[1]$, $f_j\in \h[1]$, and $k > 0$. Because of the compatibility condition $\phi_0 = Q_{\h[1]}$, equation \eqref{mixed} reduces to the multibrackets for $\h[1]$ when $n=0$ (c.f.\ \eqref{eq:derbr}).

Notice that $\g[1]$ is an $\li[1]$-subalgebra of $\g[1] \times \h[1]$ if and only if $\phi$ takes values   in vector fields on $\h[1]$ vanishing at the origin.
\end{remark}

\begin{defi}\label{semidir} Let $\g$ and $\h$ be finite-dimensional $\li$-algebras equipped with a compatible formal $\li$-action $\phi$ of $\g$ on $\h[1]$. The associated $\li$-algebra structure on $\g \times \h$ is the \emph{extension} of $\g$ by $\h$ via $\phi$, denoted
  $\g \ltimes \h$.  We say that this extension  is a \emph{semidirect product} when $\g[1]\times \{0\}$ is   an $\li[1]$-subalgebra.
\end{defi}

 Notice that equivalent compatible actions  deliver  
extensions 
 which are $\li$-isomorphic, by Proposition \ref{prop:isoQM}.

The notion of extension of $\li$-algebras was introduced in terms of coderivations by Baez and Fr\'egier \cite{BaezFregier}, who use semidirect products and cocycles in the $\li$-category to construct such extensions.

\begin{ep}[Lie algebra extensions, see also \S \ref{sec:LA}]\label{ex:la}
	In the special case where $\g$ and $\h$ are ordinary Lie algebras, any $\linfty$-algebra extension is necessarily a Lie algebra (since it is concentrated in degree $0$). On the other hand, any compatible $\linfty$-action of $\g$ on $\h[1]$ is given by a pair $(\sigma, \psi)$, where $\sigma \colon \g \to \Der(\h)$ and $\psi: \wedge^2 \g \to \h$, such that
 \begin{align}  \label{eq:nonab}
\sigma([x,y]_{\g})-[\sigma(x),\sigma(y)] + ad_{\psi(x,y)} &= 0
 \text{ for all }& x,y \in \g,\\
\label{eq:nonab2}\psi(x,[y,z]_{\g})
+[\sigma(x),\psi(y,z)]+c.p.&=0  \text{ for all }& x,y,z \in \g.
\end{align} 
The higher components of the $\linfty$-action necessarily vanish by degree considerations. 

The equations \eqref{eq:nonab}--\eqref{eq:nonab2} are exactly those that define a nonabelian $2$-cocycle \cite{hoch} on $\g$ with values in $\h$.  Thus, in this special case, Proposition \ref{sdli} reduces to the well-known fact that Lie algebra extensions $\h \to \g \times \h \to \g$ are in correspondence with such nonabelian $2$-cocycles. The connection between such extensions and compatible $\linfty$-actions  was already observed in \cite[Prop. 2.7]{shengzhu2}.

Maps $b \colon \g \to \h$  act on the space of pairs $(\sigma, \psi)$ as in Def. \ref{def:equiv}, and the action agrees with the one described explicitly in  \cite[\S 5]{MichorExtLA}.
\end{ep}

In the following subsections, we consider some other examples of extensions.

 \subsection{\texorpdfstring{$L_{\infty}$}{L-infinity}-modules}\label{sec:modules}
In this subsection, we observe that the notion of \emph{$\linfty$-module}, due to Lada and Markl \cite{LadaMarkl}, can be seen as a special case of $\linfty$-actions.

Let $(W[1],\partial)$ be a \emph{differential graded vector space}. In other words, $W[1]$ is a graded vector space, and $\partial$ is a linear differential on $W[1]$. Then $\End(W[1])$ is a DGLA with the graded commutator bracket $[A,B]_{\comm}=AB-(-1)^{|A||B|}BA$ and differential $  
[\partial,\cdot]_{\comm}$. According to Lada and Markl \cite[Theorem 5.4]{LadaMarkl}, a \emph{$\g$-module structure} on $W$[1] is equivalent to an $\linfty$-algebra morphism 
\begin{equation*}
\Phi \colon \g \rightsquigarrow (\End(W[1]),[\cdot,\cdot]_{\comm},[\partial,\cdot]_{\comm}).
\end{equation*}
We may identify $\End(W[1])$ with  the space 
of linear vector fields on  $W[1]$, by mapping $A$ to the unique vector field $Y_A$ such that $[Y_A,\iota_w]=\iota_{Aw}$ for all $w\in W[1]$.
This determines an isomorphism of DGLAs
\begin{equation}\label{endvf}
\left(\End(W[1]),[\cdot,\cdot]_{\comm},[\partial,\cdot]_{\comm}\right)\;\cong\;
\left(
\vect_{\lin}(W[1]),[\cdot,\cdot],[Y_{\partial},\cdot]
\right).
\end{equation}

Hence we can view the $\g$-module structure as being an example of an $\linfty$-action of $\g$ on $W[1]$ compatible with  $-Y_{\partial}$.  {(The minus sign arises from the d\'{e}calage isomorphism; see the paragraph before Definition \ref{def:action}.)} On the other hand, any $\linfty$-action that is compatible with $-Y_{\partial}$ and \emph{linear} (in the sense that the vector fields on $W[1]$  {in the image of the action map} are linear) comes from a $\g$-module structure. To summarize, we have the following:

\begin{prop}\label{linvf} Let $\g$ be a finite dimensional $\linfty$-algebra and
 $(W[1],\partial)$ a differential graded vector space. There is a  one-to-one correspondence between
 \begin{enumerate}
\item $\g$-module structures on $W[1]$ 
\item linear $\li$-actions of $\g$ on  $W[1]$  that are compatible with $-Y_{\partial}$.
\end{enumerate}
\end{prop}

  Now suppose that $W$ is finite-dimensional. Given a linear $\linfty$-action of $\g$ on  $W[1]$, the extension $\g \ltimes W$ only has nonzero multibrackets of the following types:
\begin{itemize}
	\item[(a)] multibrackets on $\g$, and 
     \item[(b)] mixed brackets of the form $\wedge \g \otimes W \to W$, of which the zero component $W \to W$ coincides with the differential $-\partial$. More precisely, identifying $\End(W)=\End(W[1])$, we have 
 for all $v_i\in \g$ and $w\in W$:
 $[v_1,\dots,v_k,w]=(-1)^{|v_1|+\dots+|v_k|}\Phi(v_1\wedge \dots\wedge v_k)w$.
 \end{itemize}
In this case, the generalized Jacobi identity in Definition \ref{li} reduces to the equation in \cite[Definition 5.1]{LadaMarkl} that forms the original definition of $\g$-modules in terms of multibrackets. Thus, the linear version of Proposition \ref{sdli} allows us to recover the correspondence of \cite[Theorem 5.4]{LadaMarkl}.

\subsection{The adjoint module of an \texorpdfstring{$L_{\infty}$}{L-infinity}-algebra}
In this subsection we make use of Proposition \ref{sdli} to derive the definition   of adjoint representation in the  $\li$-category.
Let $\g$ be a finite dimensional $\linfty$-algebra, and let $Q_{\g[1]}$ be the associated homological vector field on $\g[1]$. Let $Q_{T\g[1]}$ denote the complete lift (in the sense of Yano and Ishihara \cite{yano-ishi}) of $Q_{\g[1]}$ to the tangent bundle $T\g[1]$. That is, if the linear functions on $T\g[1]$ are identified with $1$-forms on $\g[1]$, then $Q_{T\g[1]}$ is the unique vector field on $T\g[1]$ whose action on linear functions coincides with the Lie derivative $\lie_{Q_{\g[1]}}$.
 
To make things more explicit, choose linear coordinates $\xi^i$ on $\g[1]$, and let $\tilde{\xi}^i$ denote the corresponding fiber coordinates on $T\g[1]$. Write $Q_{\g[1]} = Q_{\g[1]}^i \pdiff{}{\xi^i}$. Then
\begin{equation}\label{Tgcoord}
	Q_{T\g[1]} = Q_{\g[1]}^i \pdiff{}{\xi^i} + \tilde{\xi}^j \pdiff{Q_{\g[1]}^i}{\xi^j} \pdiff{}{\tilde{\xi}^i}.
\end{equation}
The complete lift preserves degree and Lie brackets, so $Q_{T\g[1]}$ is a (formal) homological vector field, 
{giving an $\li$-algebra structure on $T\g$, which is an extension of $\g$. The kernel, which we denote $\tilde{\g}$, is isomorphic to $\g$ as a graded vector space, but has a different $\li$-algebra structure. Specifically, it can be seen from the coordinate description \eqref{Tgcoord} that the $1$-bracket on $\tilde{\g}$ is the same as that on $\g$, but all the higher brackets vanish, so that $\tilde{\g}$ is a differential graded vector space.}

{We may canonically identify $T\g$ with $\g \times \tilde{\g}$ as graded vector spaces. By Proposition \ref{sdli}, the $\li$-algebra extension
\begin{equation}\label{tangentext}
	\tilde{\g} \to T\g = \g \ltimes \tilde{\g} \to \g
\end{equation}
corresponds to an $\li$-action of $\g$ on $\tilde{\g}[1]$. Furthermore, since $Q_{T\g[1]}$ is a linear vector field, we have that the action is linear. By Proposition \ref{linvf}, we obtain:}
\begin{prop}\label{iadj}
	Let $\g$ be a finite-dimensional $\linfty$-algebra, and denote its $1$-bracket by $\partial$. Then  the  differential graded vector space $(\tilde{\g}[1],-\partial)$ naturally has the structure of an $\linfty$-module (see \S\ref{sec:modules}) over $\g$, whose  corresponding  extension is the $\linfty$-algebra $T\g$.
\end{prop}
We refer to  $\tilde{\g}[1]$ as the \emph{adjoint module} of $\g$.

From the coordinate description \eqref{Tgcoord} and the derived bracket formula \eqref{eq:derbr}, we can directly see that 
the map $\phi_X \colon \Sym (\g[1]) \to \vect(\tilde{\g}[1])[1]\cong \End(\tilde{\g}[1])[1]$
is given by $\phi_X(e_1,\dots,e_k)=\{e_1,\dots,e_k,\bullet\}$.
Applying the  d\'ecalage isomorphism
\eqref{deca} we obtain the $\linfty$-module structure of Proposition \ref{iadj}, which (identifying $\End(\tilde{\g})$ with $\End(\tilde{\g}[1])$) reads
$$\wedge^+ \g \to \End(\tilde{\g}),\;\;\; v_1\wedge \dots \wedge v_k \mapsto (-1)^{|v_1|+\dots+|v_k|}[v_1,\dots v_k,\bullet].$$
Of course, when $\g$ is an ordinary Lie algebra, we recover the usual adjoint representation $\g \to \End(\g)$, $v\mapsto [v,\bullet]$.

\subsection{Representations up to homotopy of Lie algebras}

Let $\g$ be an ordinary finite-dimensional Lie algebra. In this subsection we observe that $\linfty$-modules over $\g$ are in correspondence with representations up to homotopy of $\g$, in the sense of \cite{rep-hom, gm:vba}. We note that this correspondence is essentially a special case of \cite{m:lamods}.

Let $(E,\partial)$ be a finite-dimensional differential graded vector space with an $\li$-module structure over  $\g$. By  Proposition \ref{linvf} the module structure corresponds to a linear $\li$-action of $\g$ on  $E$, compatible with the vector field $-Y_{\partial}$ on $E$ associated to $\partial$. By Theorem \ref{charvf}, such $\li$-actions correspond to degree $1$ elements 
$$X\in S(\g[1]^*) \otimes \vect_{\lin}(E)$$ 
whose zero component is $-Y_{\partial}$,
such that $X + Q_{\g[1]}$ is
a  homological vector field    on $\g[1]\times E$.
Notice that such an $X$ can be regarded as an element 
$$\omega\in S(\g[1]^*) \otimes \End(E)$$ 
via the identification \eqref{endvf}. The condition that  $X + Q_{\g[1]}$ be homological is therefore  equivalent to
$D^2=0$ for the operator $D := Q_{\g[1]} +  \omega$ on $S(\g[1]^*) \otimes E$. The latter is exactly the definition \cite{rep-hom, gm:vba} of a representation up to homotopy of $\g$ on $E$. 

Thus we have the following. 
\begin{prop} \label{prop:rephom}
	There is a one-to-one correspondence between  
	 \begin{enumerate}
	\item $\linfty$-modules $(E,\partial)$ over $\g$
	\item representations up to homotopy of $\g$ on $(E,-\partial)$.
\end{enumerate}
 An $\li$-module 
	given by maps $f_n \colon \wedge^n \g \to \End_{1-n}(E)$ ($n\ge 1$) 
	corresponds to the representations up to homotopy $\omega$ with components $\omega_n \in \wedge^n \g^*\otimes End_{1-n}(E) $ given by 
\begin{equation*}
f_n(v_1,\dots,v_n) = (-1)^{[\frac{n}{2}]} \iota_{v_n} \cdots \iota_{v_1} \omega_n.
\end{equation*}
\end{prop}
\begin{proof}
It remains to prove the formula relating $f_n$ to $\omega_n$.
Since the d\'ecalage isomorphism \eqref{deca} does not introduce additional signs, we have
 \begin{align*}
f_n(v_1,\dots,v_n) = 
\phi_X (z_1\cdots z_n) &= [[\dots [X, \iota_{z_1}],\dots],\iota_{z_n}]|_0 \\
&= (-1)^{[\frac{n}{2}]}  [\iota_{z_n},[ \dots ,[\iota_{z_1}, X]\dots]]|_0 
= (-1)^{[\frac{n}{2}]} \iota_{v_n} \cdots \iota_{v_1} \omega_n.
\end{align*}
	where we write $z_i:=v_i[1]$ for $v_i\in \g$.
\end{proof}

\begin{remark}
     It was noted in \cite{rep-hom} that a representation up to homotopy of $\g$ on $W^*$ induces an $\linfty$-algebra structure on $\g \times W$, which ``deserves the name semidirect product.'' Proposition \ref{prop:rephom} (together with Proposition \ref{sdli}) can be interpreted as a justification of their statement, demonstrating that $\g \times W$ is indeed a semidirect product in the $\linfty$ category.
\end{remark}

\section{Extensions of Lie algebroids}\label{sec:LA}
In this section we consider $\li$-actions of Lie algebras on degree 1 $Q$-manifolds (see also \cite{ZZL}\cite{ZZL2}). This allows us to extend Example \ref{ex:la} to the case of Lie algebroids.

Let $\g$ be a Lie algebra. Let $(A,[\cdot,\cdot]_A,\rho_A)$ be a Lie algebroid over $M$, so that $A[1]$ is a degree $1$ $Q$-manifold, with homological vector field $Q_{A[1]}$ given by the Lie algebroid differential. 

Denote by $\CDO(A) \to M$ the Lie algebroid whose sections are covariant differential operators \cite{MK2} on the vector bundle $A \to M$. Let $\Gamma_A(\CDO(A))$ denote the subspace of covariant differential operators $Y$ that respect the Lie algebroid structure, in the sense that
$$Y[a,b]_A= [Ya,b]_A+[a,Yb]_A,\;\;\;\;\;\; \underline{Y}(\rho_A(a)f)=\rho_A(Ya)f+\rho_A(a)(\underline{Y}f)$$ 
for all $a,b\in \Gamma(A)$ and $f\in C^{\infty}(M)$. Here ${\underline{Y}}\in \vect(M)$ is the symbol of $Y$. 

Then $\li$-actions of $\g$ on $A[1]$ compatible with $Q_{A[1]}$ are given by linear maps
\begin{align*}
\sigma \colon& \g \to \chi_0(A[1])=\Gamma_A(\CDO(A))\\
\psi \colon& \wedge^2 \g \to \chi_{-1}(A[1])=\Gamma(A),
\end{align*}
such that analogues of \eqref{eq:nonab} and \eqref{eq:nonab2} are satisfied. 

By Theorem \ref{charvf}, we obtain a Lie algebroid structure on the vector bundle $(\g\times M)\oplus A \to M$, which we denote by $\g \ltimes A$. Explicitly, the anchor is  $(x,a)\mapsto \underline{\sigma(x)}+\rho_A(a)$ for all $x\in \g$ and $a\in A$, and the bracket is given by  
\begin{equation*}
\big[(x_1,a_1)\;,\;(x_2,a_2)\big]_{\g \ltimes A}=\big( [x_1,x_2]_{\g}\;,\;
[a_1,a_2]_{A}+\sigma(x_1)a_2-\sigma(x_2)a_1+\psi(x_1,x_2) \big).
\end{equation*}   
Theorem \ref{charvfQM} shows that $\g \ltimes A$ fits into an exact sequence of Lie algebroids
$A \to \g \ltimes A \to \g$, and that all Lie algebroid extensions of $\g$ by $A$ arise as above.   We summarize the above discussion:

\begin{prop}\label{sdla}Let $\g$ be a Lie algebra and $A$ a Lie algebroid over $M$.
There is a one-to-one correspondence between
\begin{enumerate}
\item pairs of maps $(\sigma,\psi)$ satisfying the analogues of eq. \eqref{eq:nonab} and \eqref{eq:nonab2}
\item  Lie algebroid extensions  $A \to (\g\times M)\oplus A \to \g$.
\end{enumerate}
\end{prop}
 
\begin{remark}  See   \cite[\S 2]{MichorExtLA} for an interpretation of the above proposition in terms of splittings in the Lie algebra case.
See \cite[\S 4.5]{MK2}\cite{BrahicExt} for a general theory of Lie algebroid extensions (in which $\g$ is allowed to be any Lie algebroid).
\end{remark}
 
\begin{ep} Let $G$ be a Lie group integrating $\g$, and consider a group action of $G$ on $A$ by Lie algebroid automorphisms. Differentiating the action we obtain a map $\sigma$ as above which moreover preserves Lie brackets. (To see this  notice that  the Lie algebroid $\CDO(A)$ integrates to the Lie groupoid over $M$ whose arrows are linear  isomorphisms between the fibers of $A$). Hence, setting $\psi=0$, by Proposition \ref{sdla} we obtain a Lie algebroid extension of $A$ and $\g$.

In particular, when $A$ is the trivial Lie algebroid over $M$,   we recover the transformation algebroid $\g \ltimes M$ as in Remark \ref{rem:LA}.
\end{ep}
 
We now comment on equivalences.
\begin{prop}  
Two compatible $\li$-actions of $\g$ on $A[1]$  are equivalent (in the sense of Definition \ref{def:equivQM}) if and only if the associated
Lie algebroid structures on $(\g \times M)\oplus A$ are isomorphic by isomorphisms that commute with the maps appearing in $A \to (\g \times M)\oplus A \to \g$.
 \end{prop}
 
\begin{proof}
One implication is given by Proposition \ref{prop:isoQM}, so it remains to prove that the existence of an isomorphism of Lie algebroids $\psi$ as above implies that the corresponding compatible $\li$-actions are equivalent.
Consider $\psi \colon (\g\times M)\oplus A \to (\g\times M)\oplus A$,  and 
denote by $\psi^*$ the induced automorphism on sections of the exterior algebra of $((\g\times M)\oplus A)^*$.
The map $\psi$
 must be of the form $\psi(x,a)=(x,a+\phi(x))$ for a linear  $\phi \colon  \g \to \Gamma(A)$. Let $\lambda\in (S^+(\g[1]^*)\otimes \vect(A[1]))_0$ be the corresponding vector field on $\g[1]\times A[1]$, defined by $\iota_{\phi(v)[1]}=[\lambda,\iota_{v[1]}]\in \chi_{-1}(A[1])$. Then one computes 
\begin{align*}
\psi^*(\eta)&=\eta=e^{-\lambda}(\eta)\;\;\;\;\;\;\;\;\;\;\;\;\;\;\;\;\text{ for all }\eta\in \g[1]^*=C_1(\g[1])\\
\psi^*(\xi)&=\xi+\phi^*(\xi)=e^{-\lambda}(\xi)\;\;\;\;\text{ for all }\xi\in \Gamma(A[1]^*)=C_1(A[1]),
\end{align*}
implying that  the  diffeomorphism  $\Psi$ of $\g[1]\times A[1]$ 
corresponding to $\psi$ satisfies $\Psi^*=e^{-\lambda}$ (where $\Psi^*$ denotes the ``pullback of functions''). In other words, $\Psi$ is the time-1 flow of the vector field $-\lambda$.
 
Consider the Lie algebroid differentials $d_1,d_2$ encoding the two Lie algebroid structures on $(M\times \g)\oplus A$, and view them as homological vector fields $Q_1,Q_2$ on $\g[1]\times A[1]$. 
The fact that $\psi$ is a Lie algebroid isomorphism means that
$\Psi_*(Q_1)=Q_2$. Here $\Psi_*$ denotes the push-forward of vector fields on $\g[1]\times A[1]$, which by the above considerations we can write as $e^{ad_{\lambda}}$. Hence $e^{ad_{\lambda}}(Q_1)=Q_2$, which by \eqref{eqn:mcequivalent} means that the two compatible $\li$-actions of $\g$ on $A[1]$    are equivalent.
\end{proof}

\begin{ep}\label{ex:map}
Let $\phi \colon \g \to \Gamma(A)$ be any linear map. Then $$\sigma(x):=[\phi(x),\cdot]_A,\;\;\;\;\;\;\psi(x\wedge y)=[\phi(x),\phi(y)]_A-\phi([x,y]_{\g})$$
defines an $\li$-action of $\g$ on $A[1]$ compatible with $Q_{A[1]}$.
 
This $\li$ action is obtained twisting the trivial action by means of the linear map $\phi$, as in \S \ref{sec:equiva}. In particular, the  Lie algebroid structure we obtain on  $(\g\times M)\oplus A$ is isomorphic to the  product Lie algebroid structure. A concrete isomorphism is given by $(x,a_m)\mapsto (x,(a+\phi(x))_m)$ for all $x\in \g$, $a_m\in A_m$.  
\end{ep}  

\begin{remark}
A geometric situation in which the map $\phi \colon \g \to \Gamma(A)$ emerges canonically is the following. Let $G$ be a Poisson-Lie group, $(M,\pi)$   a Poisson manifold, and let $A:=T^*M$ be the cotangent algebroid. Let $G\times M \to M$ a Poisson action with (not necessarily equivariant) moment map\footnote{When the Poisson structure on $G$ is zero, this recovers the usual notion of moment map into $\g^*$.} $J \colon M \to G^*$ \cite[\S 3.3]{Lu}. This means that $J$ satisfies
  $\pi(J^*(x^l))=x_M$ for all $x\in \g$, where $x^l$ is the left-invariant 1-form on $G^*$ whose value at the identity is $x$, and $x_M$ the vector field on $M$ given by the action.
 Then one can simply define $\phi(x)=J^*(x^l)$, and by Example \ref{ex:map} obtain a canonical Lie algebroid structure on $(\g\times M)\oplus A$.
 
 Notice that, for \emph{any} Poisson action, Jiang-Hua Lu \cite[\S 4]{Lu:PoisAlgoid} defined a Lie algebroid structure on the same vector bundle. Lu's construction does not require a moment map, whereas the construction that we have described depends on the choice of moment map, so is obviously quite different.  In Lu's Lie algebroid, the mixed brackets take values in both summands, while in ours they take values only in the second summand $A$.
\end{remark}

\bibliographystyle{habbrv}
\bibliography{bibrajmarco}
\end{document}